\documentclass{amsart}
\usepackage[a4paper, left=3cm, right=3cm]{geometry}
\usepackage{tikz}
\usetikzlibrary{calc,arrows}
\usepackage{amsthm,amsmath,amssymb}
\usepackage{natbib}
\usepackage{graphicx}
\newtheorem{theorem}{Theorem}
\newtheorem{lemma}[theorem]{Lemma}

\newtheorem{corollary}[theorem]{Corollary}
\newtheorem{remark}[theorem]{Remark}

\def\Erl{{\rm Erlang}}
\def\E{{\mathbb E}}

\def\N{{\mathbb N}}
\def\Pr{{\mathbb P}}

\newcommand{\qix}{l}

\newcommand{\geqSt}{\geq_{\mbox{st}}}

\DeclareMathOperator*{\argmin}{arg\,min}

\begin{document}

\title[Pure threshold strategies]{Pure threshold strategies for a two-node tandem network under partial information}

\author{Bernardo D'Auria}
\address{Statistics Department, Madrid University Carlos III, Legan\'es, Spain \\ ICMAT, Madrid University Carlos III, Legan\'es, Spain}
\email{bernardo.dauria@uc3m.es}
\thanks{The research of the first author is partially supported by the Spanish Ministry of Education and Science Grants MTM2013-42104-P (FEDER funds), MTM2010-16519 and RYC-2009-04671. The authors would like to thank the anonymous referee for the insightful comments that helped in improving the correctness of the statement of the result.}

\author{Spyridoula Kanta}
\address{Statistics Department, Madrid University Carlos III, Legan\'es, Spain}

\begin{abstract}
In a two node tandem network, customers decide to join or balk by maximizing a given profit function whose costs are proportional to the sojourn time they spend at each queue.
Assuming that their choices are taken without knowing the complete state of the system, we show that a pure threshold equilibrium policy exists. 
In particular we analyze the case when the partial information consists in informing the arrival customers of the total number of users in the network.
\end{abstract}


\date{}

\maketitle

\section{Introduction}\label{sc:introduction}
Queueing literature is recently devoting an increasing attention to the economic analysis of queueing systems.
Indeed in real applications it is not uncommon that the input to a queueing system is not exogenously defined and is the result of the combined effect of the decisions made by the arriving customers.
They may decide whether to join or balk the system according to their convenience and these choices in general lead to a final equilibrium. 
This phenomenon is mathematically modeled by assuming they are rationally optimizing a given individual profit function.
This research, started in the '70s by \cite{naor:1969} and \cite{edelson:hildebrand:1975}, now has reached a good maturity, two central monographs are \cite{hassin:haviv:2003} and \cite{stidham:2009}.
Most of the literature focuses on a single server system, while we focus here on network models, in particular a series of two $M/M/\cdot$ queues.
Previous studies have looked at parallel queues \cite{whitt:1986}, \cite{hassin:1996} and \cite{haviv:zlotnikov:2011} and for more general topologies extensive studies have been done in the field of telecommunications, see \cite{courcoubetis:weber:2003,roughgarden:2005}. 
A close model is \cite{burnetas:2011}, where a series of queues of $M/M/m$ types is analyzed and the form of the symmetric customer equilibrium is derived together with the explicit socially optimal strategies. 
The main difference with our model is that there customers make their decisions without getting any information on the state of system, while here they know the total number of customers already inside.
Usually network models show an intrinsic difficulty in getting explicit results, and this partially explains a relatively scarcer literature.
The two node tandem network that we study has the advantage of being simpler and allowing a complete analysis. 
Customers make the decisions to balk or join after knowing how many customers are already in the network. 
In real applications, it is common that  people do not know the complete information on the state of the system, 
as usually this information is shortly summarized to simplify the decision process.
Examples may be found  in healthcare systems, where treatment requires two different steps, such as a first queue to get a doctor reservation and a second queue to be attended by  the doctor.
The interesting result is that the partial information setting simplifies drastically the analysis, allowing  to get for this specific case explicit results.

\par
The model is introduced in Section \ref{sc:model},
we compute in Section  \ref{sc:fully-observable}  the expected sojourn time of an arriving customer assuming that the full state of the system is known.
In Section \ref{sc:dist-total-queue} the same analysis is done when the arriving customers are informed about  the total number of customers in the network.
Finally  we compute the equilibrium strategy in  Section \ref{sc:partial-total-queue}.

\section{The model}\label{sc:model}
We consider a tandem network 
with two single server nodes with infinite buffers and service times independent and exponentially distributed. 
Using the index $\qix$, with {$\qix=1$ or $2$}, to refer to the first or the
second node, we denote by $\mu_\qix$ the service rate at node $\qix$.
Customers arrive to the system according to a Poisson process with rate $\lambda$ and before joining the network  they receive partial information about the state of the system.
The state space is $\N^2$, that is all possible pairs $(Q_1, Q_2)$ 
with $Q_\qix$ the queue length at node $\qix$. 
A \emph{tagged customer} that just arrives,  gets a reword $R$ for joining the network, and pays for each unit of sojourn time at node $\qix$ a cost $C_\qix$ with a resulting random profit given by
$P = R - C_1 \, S_1 - C_2 \, S_2 \ ,$
where $S_\qix$ denotes the sojourn times she would spend at node $\qix$.

The tagged user makes her decision by optimizing the expected profit given the information she receives at her arrival time, $k=Q_1+Q_2$, that is
\begin{equation}\label{eq:profit}
P_K(k) = R - C_1 \, T_{K,1}(k) - C_2\, T_{K,2}(k) \ ,
\end{equation}
with $T_{K,\qix} = \E_K[S_\qix|Q_1 + Q_2=k]$. The subindex $K$ tells that the rest of the population is using a pure threshold strategy with threshold $K$ in joining the queue. That is all users besides the tagged one join the network if and only if it contains less than $K$ customers. 

The main result of the paper is to show that the tandem network admits a pure threshold $K$,  that is there exists a $K\in\N$ such that
$$
P_K(k) \geq 0 \mbox{ as } k < K 
\quad \mbox{ and } \quad 
P_K(k) < 0 \mbox{ as } k \ge K \ .
$$
\begin{remark} 
By using the subindex $K$, we are implicitly assuming that the rest of population is not allowed to use strategies different from a pure threshold one. This assumption is not restrictive for our purposes, but it does not preclude the existence of policies (even of equilibrium type) that are of a different form.
\end{remark}
\begin{remark} 
We always assume that $R>C_1/\mu_1+C_2/\mu_2$. Being this relation false, a user would get negative net profit even joining an empty network implying a unique equilibrium given by the empty system.
\end{remark}

Before analyzing the described model, we first study the case when the complete information is available to the arriving customers. This is done in the next section.
\section{Mean sojourn times}\label{sc:fully-observable}
Let  $S_\qix(n,m)$ be the sojourn time spent at queue $\qix$ by a tagged customer  
that joins a system being in state $(n-1,m)$, that is she is going to occupy position $n$ in the first queue.
Let $T_\qix(n,m)=E[S_\qix(n,m)]$ be the corresponding expectation and  
$T(n,m)=T_1(n,m)+T_2(n,m)$ the total expected sojourn time.
The sojourn time in the first queue is Erlang distributed, that is $S_1(n,m) \sim \Erl(n,\mu_1)$ with mean 
$\label{def:T_1}
T_1(n,m)= n / \mu_1 \ .
$
The total sojourn time can be computed recursively by applying a first step analysis, 
that leads to the following formula, 
\begin{equation}\label{difference equation T(n,m)}
T(n,m)=\frac{1}{\mu_1 + \mu_2}+\frac{\mu_1}{\mu_1 +
\mu_2}T(n-1,m+1)+\frac{\mu_2}{\mu_1 + \mu_2}T(n,m-1),\ \
n,\, m > 0.
\end{equation}
The second term in the right hand side of \eqref{difference equation T(n,m)} considers a potential departure from the first queue 
and the last term a potential departure from the second queue. These events occur  with probability $\mu_\qix/(\mu_1+\mu_2)$, $\qix=1,2$ respectively.
To complete the recursion the following boundary conditions are needed
\begin{equation}\label{T(0,m) and T(m,0)}
T(0,m)=\frac{m}{\mu_2} \ ; \quad 
\quad 
T(n+1,0)=\frac{1}{\mu_1}+T(n,1) 
,\ \
n,\, m > 0.
\end{equation}
Using \eqref{difference equation T(n,m)} we  get a recursive formula to compute $T_2(n,m)$ as shown in the following lemma.
\begin{lemma}\label{lm:rec:T2}
 The expected sojourn time at the second queue, $T_2(n,m)$, can be computed with the following 
recursive formula
\begin{equation}  \label{rec T_2(n,m)}
T_2(n,m) =   \left( \frac{\mu_2}{\mu_1 + \mu_2}\right)^m T_2(n-1,1) 
+  \frac{\mu_1}{\mu_1 + \mu_2} \sum_{k=0}^{m-1} \left( \frac{\mu_2}{\mu_1 + \mu_2}\right)^k 
	T_2(n-1,m+1-k) 
\end{equation}
valid for $n>0$ and $T_2(0,m) = m / \mu_2$, with $m\geq0$.
\end{lemma}
\begin{proof}
By \eqref{difference equation T(n,m)}, we get that $T_2(n,m)$ satisfies the following recursive equation 
$$T_2(n,m) =  \frac{\mu_1}{\mu_1 + \mu_2}T_2(n-1,m+1)+\frac{\mu_2}{\mu_1 + \mu_2}T_2(n,m-1),\ \
n,\, m > 0, $$
and by \eqref{T(0,m) and T(m,0)}, similar boundary conditions are satisfied. By induction argument it is then straightforward to verify that
\eqref{rec T_2(n,m)} holds true.
\end{proof}
The following lemma characterizes the conditions under which the $T$-functions are monotone non-decreasing in the variable $n$.
These conditions are important for the analysis of Section \ref{sc:partial-total-queue}.
\begin{lemma}\label{lm:dec:T}
The functions $T_1(n,m)$ and $T(n,k-n)$ are  non decreasing in $n$.
The function $T_2(n,m)$ is non decreasing in $n$ if and only if $\mu_1 \geq \mu_2$.
\end{lemma}
\begin{proof}
The statement is obvious for $T_1(n,m)$ that does not depend on $m$.

One way to show that the function $T(n,k-n)$ is non decreasing in $n$, for $n \leq k$ is by proving that $T(n+1,m) \geq T(n,m+1)$ by induction using equations (\ref{difference equation T(n,m)})--(\ref{T(0,m) and T(m,0)}). 
We prefer to use a coupling argument. 
Using the same probability space, we construct two networks starting respectively with 
$(n+1,m)$ and $(n,m+1)$ initial users. The proof follows by comparing 
the waiting times of the customers that are the last ones in the first queue of both networks,
and showing that the one in the former network waits more than the corresponding one in the latter.
To construct the coupling we assume that the service times for all customers are the same in both
networks but we move the customer in service at the first queue of the first network at the end of the queue of the second node of  the second network. Since the exit times are ordered by the FIFO discipline and because the moved customer reduces its sojourn time by her service time in the first node, the result holds. 

Finally to show that $T_2(n,m)$ is non decreasing in $n$ we prove that  $\Delta_1 T_2(0,m) \geq 0$ for all $m$ where
$\Delta_1 T_2(n,m) = T_2(n+1, m) - T_2(n,m)$.
From \eqref{rec T_2(n,m)}, the following holds for any $n>0$ and $m\geq0$,
\begin{equation}  \label{rec Delta T_2(n,m)}
\Delta_1 T_2(n,m) =   \left( \frac{\mu_2}{\mu_1 + \mu_2}\right)^m \Delta_1 T_2(n-1,1) 
+  \frac{\mu_1}{\mu_1 + \mu_2} \sum_{k=0}^{m-1} \left( \frac{\mu_2}{\mu_1 + \mu_2}\right)^k 
	\Delta_1 T_2(n-1,m+1-k) \ .
\end{equation}
If $\Delta_1 T_2(0,m) \geq 0$ the same holds for $n>0$ as all the coefficients in (\ref{rec Delta T_2(n,m)}) are
positive. In the opposite case $T_2(n,m)$ is clearly decreasing for some value of $(n,m)$.
Let  $\alpha = \mu_1/\mu_2$, one can check that
$$
\Delta_1 T_2(0,m) 
 = \frac{1}{\mu_2}\left(\frac{\alpha-1+(\alpha+1)^{-m}}{\alpha}\right)
\ .
$$
The quantity above is decreasing in $m$. To check that it would be non negative for any value of $m$ we take $m\to\infty$ and get the 
required condition $\alpha \geq 1$.
\end{proof}
\section{Expected sojourn times under the K-strategy}\label{sc:dist-total-queue}
We assume that all arriving customers receive the partial information about the state of the network
and decide to join only if the number of customers in the system, say $k$, is less then a given threshold $K\geq0$.
%
%
Under the $K$--strategy  the tandem network behaves as a
\emph{semiopen} Jackson network, see \citep[Section 2.3]{chen:yao:2001}. 
Let $Q^*_\qix$ be the stationary random number of customers at node $\qix$ 
and $Q^*= Q^*_1+Q^*_2$ be the stationary total number of customers
in the queue.
The stationary distribution is given by, see \citep[Theorem 2.5]{chen:yao:2001},
\begin{equation}\label{K.strat.produc.form}
 \pi_K(n,m) = \Pr_K(Q^*_1 = n , Q^*_2  = m )  = c_K \, \rho_1^n \rho_2^m,\ \ n+m\leq K
\end{equation}
where $\rho_\qix=\lambda/\mu_\qix$ and 
 $c_K^{-1}=\sum_{n+m\leq K} \rho_1^n \rho_2^m$ is the normalization constant.

Assuming $n \le K$, the conditional probability  $\Pr_K(Q^*_\qix = n | Q^* = k) = \rho_l^n \rho_{3-l}^{k-n} / \sum_{h=0}^k \rho_l^h \rho_{3-l}^{k-h}$ does not depend on $K$.
Let $p_\qix(n|k)=\Pr_K(Q^*_\qix = n| Q^* = k)$, after algebraic manipulation, we get 
\begin{equation}\label{cond.prob}
  p_\qix(n|k) =
\left\{ \begin{array}{ll}
 \mu_l^{k-n} \mu_{3-l}^n (\mu_1 - \mu_2)/(\mu_1^{k+1} - \mu_2^{k+1}) & \mu_1 \not= \mu_2 \\ \\
1/(1 + k) & \mu_1=\mu_2 
\end{array} \right. 
\ .
\end{equation}
The independence from $K$  allows to consider  the random variables 
$Q^*_\qix(k)$, $\qix\in\{1,2\}$, having distributions $p_\qix(\cdot|k)$ and not depending on the pure threshold strategy employed by all customers.
\begin{remark}
The assumption that the rest of population follows a threshold policy is necessary in order to have the steady
state distribution expressed in the form given in \eqref{cond.prob}.
\end{remark}

Let us define by  $T_{\qix}(k)=\E[S_\qix|Q^* = k]$ the expected sojourn time at queue $\qix$ of a tagged
customer that enters a system containing  $k$ customers.
\begin{lemma}\label{lem:T}
Assuming $\mu_1\not=\mu_2$, it holds that 
\begin{eqnarray}
T_1(k)
&=& \frac{1}{\mu_1-\mu_2}-\frac{k+1}{\mu_1}\;\frac{\mu_2^{k+1}}{\mu_1^{k+1}-\mu_2^{k+1}} \label{T_1(k)}\\
T_2(k)
&=& \left (1-\frac{\mu_2}{\mu_1} \right) \frac{\mu_1^{k+1}}{\mu_1^{k+1}-\mu_2^{k+1}}
\sum_{n=0}^k T_2(n+1,k-n)\left(\frac{\mu_2}{\mu_1}\right)^n  \label{T_2(k)}
\end{eqnarray}
and for $\mu_1=\mu_2$, $T_1(k)= 1/ \mu_1 (1+ k/2)$ and $T_2(k)= 1/ (k+1) \sum_{n=0}^k T_2(n+1,k-n)$. 
\end{lemma}
\begin{proof}
By definition $T_1(k)= 1/\mu_1 \, \sum_{n=0}^k (n+1) \, p_1(n|k)$, therefore \eqref{cond.prob} with $\mu_1 \not = \mu_2$ implies
\begin{eqnarray*}\label{tot inf:expression for T_1(k)}
T_1(k)
&=& \frac{1}{\mu_1}\frac{\mu_1-\mu_2}{\mu_1^{k+1}-\mu_2^{k+1}}
     \sum_{n=0}^k (n+1)\mu_1^{k-n}\mu_2^n \\
&=& \frac{1}{\mu_1} \frac{\mu_1^k(\mu_1-\mu_2)}{\mu_1^{k+1}-\mu_2^{k+1}}
\frac{\mu_1^{2+k} -(2+k) \mu_1\mu_2^{1+k}+(k+1)\mu_2^{2+k}}{\mu_1^k(\mu_1-\mu_2)^2} 
\ .
\end{eqnarray*}
Simplifying the expression above we get \eqref{T_1(k)}.
The formula for  $T_2(k)$ is obtained similarly by the expression $T_2(k)=\sum_{n=0}^k  T_2(n+1,k-n) \, p_1(n|k)$.
The results for $\mu_1=\mu_2$ can be obtained in a similar way or more directly by noticing that in this case 
the random variables $Q^*_\qix$ are discrete uniformly distributed on $\{0,1,\ldots,k\}$. One could also compute the limit of the expressions \eqref{T_1(k)} and \eqref{T_2(k)} as $\mu_1\to\mu_2$.
\end{proof}

\section{Threshold equilibrium stategy}\label{sc:partial-total-queue}
By Lemma \ref{lem:T} the expected profit of a tagged customer receiving the information $k$, given in \eqref{eq:profit},
does not depend on the strategy $K$ employed by the rest of customers. We can therefore compute it as
\begin{equation}\label{eq:profit.tot}
P(k) = R - C_1 \, T_1(k) - C_2\, T_2(k) \ .
\end{equation}
The tagged customer decides to enter only if $P(k) \geq 0$.
In the sequel we show under what conditions the expected net profit function is decreasing in $k$,
moreover since this function is constant with respect to the strategy $K$, 
we obtain that the equilibrium strategy is in addition a dominant strategy in the class of threshold strategies.
Before stating the main result  we require the following Lemma on the stochastic monotonicity of the random variables $Q^*_\qix(k)$.
\begin{lemma}\label{part-tor:stoc.monot}
 The random variables  $Q^*_\qix(k)$ are increasing stochastically ordered in $k\geq0$.
\end{lemma}
\begin{proof}
In order to show that $Q^*_\qix(k+1) \geqSt Q^*_\qix(k)$ it is enough to prove the
stronger condition given by the likelihood ratio ordering, see \cite{chen:yao:2001}.
This last condition can be checked by proving that 
\begin{equation}\label{st.or}
\Pr\{Q^*_\qix(k+1) = n+1\} \Pr\{Q^*_\qix(k) = n\} \geq \Pr\{Q^*_\qix(k+1) = n\}
\Pr\{Q^*_\qix(k) = n+1\} 
\end{equation}
for any $n \geq0$. It easy to check that (\ref{st.or}) holds as equality for any $n < k$
and is a strict inequality for $n=k$ where the second term is $0$ and therefore the result holds
true.
\end{proof}
\begin{remark} We explicitly note that the result of Lemma \ref{part-tor:stoc.monot} is different from requiring
that the variables $Q^*_\qix$ are stochastically ordered with respect to the strategy $K$.
This result refers to non conditional quantities and could be proved by a coupling argument similar to the one used in \cite{sakuma:miyazawa:2005}. 
\end{remark}

Lemma \ref{part-tor:stoc.monot} implies the following monotonicity result to the mean sojourn
functions.

\begin{lemma}\label{part-tor:T.monot}
The functions $T_1(k)$ and $T(k)$ are non decreasing for all values of the ratio
$\mu_1/\mu_2$.
The function $T_2(k)$ is non decreasing when this ratio is greater than $1$. 
\end{lemma}

\begin{proof}
The function $T_1(k)$ is non decreasing by Lemma \ref{part-tor:stoc.monot} and because $T_1(n,m)$ is non decreasing in $n$ as well.
For $T(k)$, it follows by the following
\begin{eqnarray}
T(k+1) &=& \E[T(Q^*_1(k+1),k+1-Q^*_1(k+1))] \nonumber \\
 &\geq&
  \E[T(Q^*_1(k+1),k-Q^*_1(k+1))]
 \geq
  \E[T(Q^*_1(k),k-Q^*_1(k))] = T(k) 
\end{eqnarray}
where in the last equality we used the fact that $T(n,k-n)$ is non decreasing, see Lemma \ref{lm:dec:T},
and the stochastic monotonicity of the variables $Q^*_1(k)$, proved in Lemma \ref{part-tor:stoc.monot}. 
A similar argument works for the function $T_2(k)$. Assuming $\mu_1\geq\mu_2$, we have
\begin{eqnarray}
T_2(k+1) &=& \E[T_2(k+1-Q^*_2(k+1),Q^*_2(k+1))] \nonumber \\
 &\geq&
  \E[T_2(k-Q^*_2(k+1),Q^*_2(k+1))]
 \geq
  \E[T_2(k-Q^*_2(k),Q^*_2(k))] = T_2(k) \ .
\end{eqnarray}
The first inequality follows by Lemma \ref{lm:dec:T} under the assumption on the service rates,
the second inequality follows by the monotonicity of $Q^*_2(k)$, shown in Lemma
\ref{part-tor:stoc.monot}, and  the fact that $T_2(k-m,m)$ is non decreasing in $m$  for any fixed $k>0$.
\end{proof}

\begin{corollary}\label{cor:mon.profit}
If $\mu_1>\mu_2$ or if $C_1 \geq C_2$ the expected net profit, $P(k)$, is  non increasing .
\end{corollary}
\begin{proof}
The result follows from Lemma \ref{part-tor:T.monot} if $\mu_1>\mu_2$.
If $C_1 \geq C_2$, it is enough to rewrite the profit function as 
$\label{tot inf:alt expression for ind prof}
P(k)
= R - (C_1 - C_2) \, T_1(k) - C_2 \, T(k)$,
and use the monotonicity of $T_1(k)$ and $T(k)$.
\end{proof}

Finally we state the main result that  gives the conditions to find the strategy that induces
the Nash equilibrium.

\begin{theorem}\label{thm:tot.inf}
When the information known by the arriving customers is the total number of customers in the network
only, then the equilibrium strategy is given by the threshold $K$ such that 
\begin{equation}\label{eq:pr.tot}
K = \argmin\{ k\in\N : \, P(k) < 0 \}\ .
\end{equation}
According to this strategy a tagged customer enters only if she finds less than $K$ customers
in the system. The $K$-strategy is a dominant strategy in the class of threshold strategies.
\end{theorem}
\begin{proof}
Let the index $K$ be the one obtained by formula \eqref{eq:pr.tot}, including $K=\infty$.
We  show that the strategy $K$ is the best response against itself.
The system is always ergodic, therefore with no loss of generality we assume it starts empty.
Since all customers employ the $K$-strategy the tagged user will never find more then $K$ customers in the system
and according to \eqref{eq:pr.tot}, she will follow the same strategy, and the result follows.

The actions that the tagged customer may take for the values of $k>K$ are irrelevant as she will never find the system in these states.
However if the monotonicity conditions given in Lemma \ref{part-tor:T.monot} hold,
the $K$-strategy leads to a \textit{subgame perfect equilibrium}, see \cite{hassin:haviv:2003}.

The $K$-strategy is dominant because it is  the best response to any other possible threshold
strategy. This holds because the net profit function does not depend neither on the arrival rate nor
on the threshold employed by other customers. 
Assuming that the system is working under a pure strategy with a threshold different from $K$,
if at some point in time customers start to behave selfishly, they will all adopt the $K$-strategy.
If the monotonicity conditions of Lemma \ref{part-tor:T.monot} are not satisfied, this statement uses the fact that the Markov chain is ergodic and it hits almost surely the null state.
\end{proof}

The expression \eqref{eq:profit.tot} depends on the values of the function $T_2(n,m)$ given in \eqref{T_2(k)} and as such
we cannot expect to obtain a closed formula for the equilibrium threshold strategy.
However  we can always compute it numerically by  \eqref{T_2(k)}.

\section{Conclusions}
This works analyzes the equilibrium behavior of a tandem network when customers may choose the actions of balking or joining the system by taking into account economic considerations. This paper is the first to look at queues in series, and the surprising result is that if users only receive partial information on the status of the network, in particular the total number of customers in the system, a pure strategy exists. It may be the case that such result may be extended to more general network topologies, this will be the subject of future research.

\bibliographystyle{apalike}
\bibliography{sym-biblio-30092015-arxiv}

\begin{thebibliography}{}

\bibitem[Burnetas, 2013]{burnetas:2011}
Burnetas, A. (2013).
\newblock Customer equilibrium and optimal strategies in markovian queues in
  series.
\newblock {\em Annals of Operations Research}, 208(1):515--529.

\bibitem[Chen and Yao, 2001]{chen:yao:2001}
Chen, H. and Yao, D.~D. (2001).
\newblock {\em Fundamentals of Queueing Networks}.
\newblock Springer-Verlag.

\bibitem[Courcoubetis and Weber, 2003]{courcoubetis:weber:2003}
Courcoubetis, C. and Weber, R. (2003).
\newblock {\em Pricing Communication Networks: Economics, Technology and
  Modelling (Wiley Interscience Series in Systems and Optimization)}.
\newblock John Wiley \& Sons.

\bibitem[Edelson and Hildebrand, 1975]{edelson:hildebrand:1975}
Edelson, N.~M. and Hildebrand, K. (1975).
\newblock Congestion tolls for poisson queueing processes.
\newblock {\em Econometrica}, 43:81--92.

\bibitem[Hassin, 1996]{hassin:1996}
Hassin, R. (1996).
\newblock On the advantage of being the first server.
\newblock {\em Management Science}, 42(4):618--623.

\bibitem[Hassin and Haviv, 2003]{hassin:haviv:2003}
Hassin, R. and Haviv, M. (2003).
\newblock {\em To Queue or Not to Queue: Equilibrium Behavior in Queueing
  Systems}.
\newblock Kluwer, Boston.

\bibitem[Haviv and Zlotnikov, 2011]{haviv:zlotnikov:2011}
Haviv, M. and Zlotnikov, R. (2011).
\newblock Computational schemes for two exponential servers where the first has
  a finite buffer.
\newblock {\em RAIRO - Operations Research}, 45(1):17--36.

\bibitem[Naor, 1969]{naor:1969}
Naor, P. (1969).
\newblock The regulation of queue size by levying tolls.
\newblock {\em Econometrica}, 37:15--24.

\bibitem[Roughgarden, 2005]{roughgarden:2005}
Roughgarden, T. (2005).
\newblock {\em Selfish Routing and the Price of Anarchy}.
\newblock The MIT Press.

\bibitem[Sakuma and Miyazawa, 2005]{sakuma:miyazawa:2005}
Sakuma, Y. and Miyazawa, M. (2005).
\newblock On the effect of finite buffer truncation in a two-node jackson
  network.
\newblock {\em Journal of Applied Probability}, 42(1):199--222.

\bibitem[Stidham, 2009]{stidham:2009}
Stidham, S. (2009).
\newblock {\em Optimal design of queueing systems}.
\newblock CRC Press, Boca Raton.

\bibitem[Whitt, 1986]{whitt:1986}
Whitt, W. (1986).
\newblock Deciding which queue to join: Some counterexamples.
\newblock {\em Oper. Res.}, 34(1):55--62.

\end{thebibliography}
\end{document}